\documentclass[letterpaper,11pt,reqno]{amsart}
\usepackage{graphicx,xcolor}
\usepackage{amsmath,amsthm}
\usepackage{amssymb}
\usepackage{float}
\usepackage{xcolor}
\usepackage{hyperref}
\hypersetup{colorlinks, linkcolor={red!50!black}, citecolor={green!50!black}, urlcolor={blue!50!black}}
\usepackage{enumerate}
\usepackage{mathrsfs}
\usepackage{mathtools}
\usepackage[english]{babel}

\usepackage{geometry} 
\theoremstyle{plain}
\newtheorem*{theorem*}{Theorem}
\newtheorem{theorem}{Theorem}[section]

\newtheorem{lemma}[theorem]{Lemma}
\newtheorem{proposition}[theorem]{Proposition}
\newtheorem{corollary}[theorem]{Corollary}

\newtheorem*{question*}{Question}
\newtheorem*{questions*}{Questions}
\newtheorem*{problem*}{Problem}
\theoremstyle{definition}
\newtheorem{definition}{Definition}[section]
\newtheorem{remark}{Remark}
\newtheorem*{example*}{Example}

\newcommand{\BS}[1][N]{\mathrm{BS}(1,#1)}
\def\htop{h_{\mathrm{top}}}

\usepackage{tikz}
\usetikzlibrary{patterns,positioning,arrows,decorations.markings,calc,decorations.pathmorphing,decorations.pathreplacing}
\usepackage{tikz-qtree}
\RequirePackage{pgffor} 

\title{Subshifts and colorings on ascending HNN-extensions of finitely generated abelian groups}
\date{\today}
\author{Eduardo Silva}

\email{eduardo.silva@ug.uchile.cl}


\begin{document}

\begin{abstract}
For an ascending HNN-extension $G*_{\psi}$ of a finitely generated abelian group $G$, we study how a synchronization between the geometry of the group and weak periodicity of a configuration in $\mathcal{A}^{G*_{\psi}}$ forces global constraints on it, as well as in subshifts containing it. A particular case are Baumslag-Solitar groups $\BS$, $N\ge2$, for which our results imply that a $\BS$-SFT which contains a configuration with period $a^{N^\ell}$, $\ell\ge 0$, must contain a strongly periodic configuration with monochromatic $\mathbb{Z}$-sections. Then we study proper $n$-colorings, $n\ge 3$, of the (right) Cayley graph of $\BS$, estimating the entropy of the associated subshift together with its mixing properties. We prove that $\BS$ admits a frozen $n$-coloring if and only if $n=3$. We finally suggest generalizations of the latter results to $n$-colorings of ascending HNN-extensions of finitely generated abelian groups.
\end{abstract}	
\maketitle 	


\section{Introduction}
\label{section:introduction}
The main object studied by symbolic dynamics on finitely generated groups are $G$-\textit{subshifts}, that is, closed subsets of the product space $\mathcal{A}^G$, for $\mathcal{A}$ a finite discrete alphabet and $G$ a finitely generated group, which are invariant under the action $\sigma$ of $G$ in $\mathcal{A}^G$, called the \textit{shift} and defined by
$$\sigma_g(x)_h=x_{g^{-1}h}\text{ for }x\in\mathcal{A}^G,\ g,h\in G.$$
An important family of subshifts are \textit{subshifts of finite type (SFT)}, which are those that can be described completely by a finite set of forbidden patterns.

Throughout the last decades there has been interest in the cases $G=\mathbb{Z}^d$, $d\ge 2$, where a more complex structure arises in comparison with the one-dimensional case $G=\mathbb{Z}$: questions which have been long answered in the case $d=1$ become significantly more difficult for $d\ge 2$, sometimes changing its veracity, or remaining as open problems to this day. An example of this contrast is the so called Domino problem, which asks whether, given a finite set of forbidden patterns, the corresponding SFT is empty or not. For $d=1$ this problem is decidable, meanwhile for $d\ge 2$ it is undecidable.

In recent years interest has grown in studying symbolic dynamics on other finitely generated groups, obtaining results in a general context as well as in more particular settings, such as free groups or Baumslag-Solitar groups.

Given $m,n\in \mathbb{Z}\backslash \{0\}$, we define the corresponding \textit{Baumslag-Solitar group} as given by its standard presentation
$$
\mathrm{BS}(m,n)\coloneqq \langle a,b \mid ba^mb^{-1}=a^n \rangle.
$$
These groups were first introduced (though their origin might be older) in \cite{baumslag_solitar_1962} by G. Baumslag and D. Solitar, in order to provide an example of a non-Hopfian group with two generators and one relator, namely, one which is isomorphic to one of its (proper) quotient groups. Since then these groups have gained attention in the fields of combinatorial group theory and geometric group theory as examples and counterexamples of different properties \cite{harpe_2003,meskin_1972}, and more recently in the field of symbolic dynamics \cite{aubrun_kari_2013,cyr2016distortion,esnay2020weakly}.

The subfamily of Baumslag-Solitar groups $\BS$, $N\ge 2$, is of particular interest, since it exhibits nice properties in contrast with the rest of Baumslag-Solitar groups. This subfamily actually comprises all the cases for which the Baumslag-Solitar group is solvable, and hence amenable, while still being non-abelian (the case $N=1$ gives $\mathrm{BS}(1,1)\cong \mathbb{Z}^2$). In particular, this allows us to find a F\o lner sequence which permits us to define a notion of topological entropy for $\BS$-subshifts. Actually, the above holds in a more general framework: ascending HNN-extensions of finitely generated groups, which will be our object of study throughout this paper. 

In section \ref{section:preliminaries} we introduce the basic concepts and notation used throughout this work. We define the basic objects of symbolic dynamics and different notions of periodicity, show how Baumslag-Solitar groups arise as HNN-extensions of $\mathbb{Z}$ and explain the advantages of restricting ourselves to the case of ascending HNN-extensions $G*_{\psi}$ of a finitely generated abelian group $G$. Finally we describe the structure of the Cayley graph of $G*_{\psi}$, which is formed by copies of $G$, which we call $G$-sections, connected between them by the stable letter associated to the HNN-extension.

In section \ref{section:weak_periodicity} we show how the geometry of $G*_{\psi}$ forces some rigidity for configurations in the full $G*_{\psi}$-shift which exhibit weak periodicity in the direction of a particular subgroup. Alternatively, how a synchronization between the stabilizer of a configuration and the map defining the HNN-extension forces rigidity in the global structure of the point. We show that a periodicity of
$\psi^{\ell}(H)$, for $H\leqslant G$ and $\ell\ge 0$, forces sufficiently high $G$-sections to be $p$-periodic (Theorem \ref{prop:ascending HNN extension periodicity}). Moreover, we show how having such a configuration affects a $G*_{\psi}$-subshift which contains it (Theorem \ref{thm: strong H periodicity HNN extensions in subshifts and SFTs}).

\begin{theorem}\label{thm:summary_weak_periodicity}
	Let $x\in \mathcal{A}^{G*_{\psi}}$ and suppose there exists a subgroup $H\leqslant G$ with $\psi(H)\leqslant H$ and $\ell\ge 0$ such that $\psi^{\ell}(H)\leqslant \mathrm{Stab}(x)$. Then
	\begin{enumerate}
		\item all $G$-sections of level greater or equal than $\ell$ are $H$-periodic,
		\item if $X$ is a $G*_{\psi}$-subshift which contains the configuration $x$, then there exists $y\in X$ for which all $G$-sections are $H$-periodic, and 
		\item if $X$ is further assumed to be an SFT and $H=G$, then $y$ can be chosen to be strongly periodic.
	\end{enumerate}
\end{theorem}

In particular for the Baumslag-Solitar group $\BS$, $N\ge 2$, Theorem \ref{thm:summary_weak_periodicity} implies that any configuration $x\in \mathcal{A}^{\BS}$ such that $\sigma_{a^{pN^\ell}}(x)=x$ has sufficiently high $p$-periodic $\mathbb{Z}$-sections. Moreover, any $\BS$-subshift $X$ containing such a configuration contains one in which all $\mathbb{Z}$-sections are $p$-periodic, and if $X$ is an SFT and $p=1$, then it contains a strongly periodic configuration with monochromatic $\mathbb{Z}$-sections.

Given a group $G$ generated by a finite subset $S\subseteq G$ it is interesting to study the dynamical properties of proper colorings (in the sense used commonly in graph theory) of its Cayley graph. We define for $n\ge 2$ the \textit{graph-coloring subshift} (GCS)
$$
\mathcal{C}_{n}\coloneqq\left\{x\in \{0,\ldots,n-1\}^G\mid \text{ for all } g\in G, s\in S: \ x_{g}\neq x_{gs} \right\},
$$
that is, each configuration $x\in \mathcal{C}_{n}$ describes a proper $n$-coloring of the Cayley graph $\Gamma(G,S)$ of $G$ with respect to the generator $S$. 

Proper $n$-colorings of $G=\mathbb{Z}^2$, or more generally $G=\mathbb{Z}^d$ for $d\ge 2$, have been studied recently \cite{alon2019mixing, peled2018rigidity, ray2020proper}. In particular, in \cite{alon2019mixing} the authors study mixing properties for colorings of $\mathbb{Z}^d$, $d\ge 2$, introducing the notion of a frozen $n$-coloring of $\mathbb{Z}^d$ as an n-coloring which cannot be modified on a finite subset to create a different coloring, and prove that $\mathbb{Z}^d$ admits frozen $n$-colorings if and only if $2\le n\le d+1$.

In section \ref{section:graph_coloring_subshifts} we consider the GCS $\mathcal{C}_n$ in $\BS$ and, among other results regarding extensibility of patterns and mixing properties, we show explicit constructions of frozen $3$-colorings in $\BS$ (Theorem \ref{thm:existence_frozen_3}), and then prove the lack of frozen $n$-colorings for every $n\ge 4$ (Theorem \ref{thm:no_frozen_n_ge_4}).

\begin{theorem}\label{thm:summary_gcs_frozen_colorings}
	For $n\ge 3$, the GCS $\mathcal{C}_n\subseteq\{0,\ldots,n-1\}^{\BS}$ contains a frozen coloring if and only if $n=3$.
\end{theorem}	

Another question of interest for $n$-colorings of a group $G$ is to estimate in how many different ways one can (properly) color a finite set, that is, the amount of patterns of $\mathcal{C}_n$ with a fixed (finite) support. The answer to this question, in turn, allows us to give estimates for the topological entropy of $\mathcal{C}_n$. This is a rather difficult question even when restricting the attention to $\mathbb{Z}^2$ and in general there are no closed formulas for any of these numbers, except for some exceptions as is the case of topological entropy of $3$-colorings of $\mathbb{Z}^2$, known to be $\frac{3}{2}\log\frac{4}{3}$ \cite{Lieb:1967zz}.

The last results of this section provide bounds for the topological entropy of the GCS $\mathcal{C}_n$ in $\BS$ (Theorem \ref{thm:gcs_entropy_estimates}), as well as showing that the topological entropy of $\mathcal{C}_3$ is positive (Theorem \ref{thm:C_3_has_positive_entropy}).
\begin{theorem} \label{thm:summary_gcs_entropy}
	For $n\ge 3$, the GCS $\mathcal{C}_n\subseteq\{0,\ldots,n-1\}^{\BS}$ satisfies
	
	$$
	\log(n-2)\le\htop(\mathcal{C}_n)\le\log(n-1).
	$$
	We also have that $\htop(\mathcal{C}_3)>0$. In fact, we further have that
	\begin{enumerate}
		\item if $N$ is odd then $\htop(\mathcal{C}_3)\ge\frac{1}{2}\log 2$, and
		\item if $N$ is even then $\htop(\mathcal{C}_3)\ge\frac{1}{2(N^2+1)}\log 2$.
	\end{enumerate}
\end{theorem}

Finally in Section \ref{section: further questions} we make some comments and questions about possible generalizations of the above results to the case of $n$-colorings of general ascending HNN-extensions of finitely generated abelian groups.

\section{Preliminaries}\label{section:preliminaries}

\subsection{Symbolic dynamics on finitely generated groups}
Let $G$ be a finitely generated group and $\mathcal{A}$ be a finite (discrete) set. Given $F\subseteq G$ finite, we call and element $p\in \mathcal{A}^{F}$ a \textit{pattern}, $F$ its \textit{support}, and use the notation $\mathrm{supp}(p)\coloneqq F$. Given patterns $p,q$, we say that $p$ is a subpattern of $q$, denoted $p\sqsubseteq q$, if there exists $g\in G$ such that $p=q|_{g\cdot \mathrm{supp}(p)}$. An element $x\in \mathcal{A}^G$ is called a \textit{configuration} or \textit{coloring}, and we use the notation $x_g\coloneqq x(g)$, $g\in G$. The set $\mathcal{A}^G$ can be endowed with the product topology, which is metrizable, compact and has a base given by the \textit{cylinders}, that is, sets of the form
$$
[p]=\{x\in \mathcal{A}^G\mid x|_{\mathrm{supp}(p)}=p \}, \ p\text{ a pattern}.
$$

The topological space $\mathcal{A}^G$ together with the left $G$-action by homeomorphisms
\begin{align*}
\sigma:G&\to \mathrm{Homeo}(\mathcal{A}^G,\mathcal{A}^G)\\
g&\mapsto \sigma_g,
\end{align*}
where for every $x\in \mathcal{A}^G$, $g,h\in G$:
$$
\sigma_g(x)_h=x_{g^{-1}h},
$$
is called the \textit{full $G$-shift} and $\sigma$ is called the \textit{shift} map. Associated to this action, we define for a configuration $x\in \mathcal{A}^G$ its \textit{orbit} $$\mathrm{Orb}(x)\coloneqq\{\sigma_g(x)\mid g\in G \},$$ and its \textit{stabilizer} $$\mathrm{Stab}(x)\coloneqq \{g\in G\mid \sigma_g(x)=x \}.$$ The latter of these sets is in fact a subgroup of $G$, and we have the following relation: $|\mathrm{Orb}(x)|=|G:\mathrm{Stab}(x)|$. A configuration $x\in \mathcal{A}^G$ is said to be \textit{weakly periodic} if $\mathrm{Stab}(x)$ is not the trivial subgroup $\{e_G\}$, and \textit{strongly periodic} if $|\mathrm{Orb}(x)|<\infty$.

A subset $X\subseteq \mathcal{A}^G$ is called a \textit{$G$-subshift} if it is closed and $\sigma$-invariant, that is, for every $g\in G$: $\sigma_g(X)=X$. Equivalently, $X$ is a \textit{$G$-subshift} if there exists a family $\mathcal{F}$ of (\textit{forbidden}) patterns such that
$$
X=\{x\in \mathcal{A}^G\mid \text{ for every finite subset }F\subseteq G:\ \sigma_g(x)|_{F}\notin \mathcal{F}, \text{ for every }g\in G \}.
$$ 
If such a family can be chosen to be finite we say that $X$ is a $G$-subshift of finite type (SFT).

Given a $G$-subshift $X\subseteq \mathcal{A}^G$ described by a family $\mathcal{F}$ of (forbidden) patterns, we say that a pattern $p$ is \textit{locally admissible} if $p$ does not contain elements of $\mathcal{F}$ as subpatterns, and \textit{globally admissible} if there exists an $x\in X$ such that $x|_{\mathrm{supp}(p)}=p$.

We say that a countable group $G$ is \textit{amenable} if there exists a sequence of finite subsets $F_n\subseteq G$, $n\in \mathbb{N}$, such that for any $g\in G$:
$$
\lim_{n\to \infty}\frac{|gF_n\triangle F_n|}{|F_n|}=0.
$$
We say that $\{F_n\}_{n\in \mathbb{N}}$ is a \textit{F\o lner sequence} in $G$. For our purposes, amenability will be important because it will allow us to generalize the notion we have for topological entropy of $\mathbb{Z}$-subshifts to more general groups. See \cite[Chapter 9]{Loh_geometric} and \cite[Chapter 4]{Cellular_automata_and_groups} for further details about amenability.

In the case where $G$ is amenable, we fix a F\o lner sequence $\{F_n\}_{n\in \mathbb{N}}$ and we define the \textit{topological entropy} of a $G$-subshift $X$ as
$$
\htop(X)\coloneqq \lim_{n\to \infty}\frac{\log|\mathcal{L}_{F_n}(X)|}{|F_n|},
$$
where  $\mathcal{L}_{F_n}(X)=\{x|_{F_n}\mid x\in X\}$. This limit always exists and it does not depend on the F\o lner sequence chosen \cite[Theorem~4.38]{Kerr2016}. Intuitively, the topological entropy of a $G$-subshift $X$ represents the rate of exponential growth at which the amount of patterns with support $F_n$ we see in the elements of $X$ increases as these sets become more and more invariant.


\subsection{HNN-extensions and Baumslag-Solitar groups}
Consider a group $G$ and two isomorphic subgroups $H,K\le G$. The HNN-extension\footnote{The letters ``HNN'' stands for Higman, Neumann and Neumann, who introduced this kind of extension in \cite{HNN49}.} $G*_{\psi}$ defined below is intuitively a form of constructing a group which contains a copy of $G$, on which $H$ and $K$ are conjugate subgroups.

\begin{definition}\label{definition.hnn_extension} Consider a group $G$ with a presentation $\langle S\left|\right. R\rangle$ and two isomorphic subgroups $H,K\le G$, with $\psi:K\to H$ an isomorphism. We define the \textit{HNN-extension} of $G$ with respect to $\psi$ by the presentation
	$$
	G*_{\psi}\coloneqq \langle S\cup\{t\} \left|\right. R\cup \{t k t^{-1}=\psi(k): k\in K \} \rangle.
	$$
	We will say that $G$ is the \textit{base group}, and that $H$ and $K$ are the subgroups conjugated in $G*_{\psi}$. We also call the generator $t$ the \textit{stable letter}.
\end{definition}
Thanks to the relations of $G*_{\psi}$, for every $k\in K$ and $h\in H$ we have that $tk=\psi(k)t$ and $\psi^{-1}(h)t^{-1}=t^{-1}h$, which allows us to choose in which order the elements of $G$ appear with respect to $t$ or $t^{-1}$. By using this property, we can find a rather simple normal form for the HNN-extension $G*_{\psi}$, as the next proposition shows.

\begin{proposition}[Britton's Lemma {\cite[Theorem~IV.2.1]{lyndon_schupp_1977}}] \label{prop:hnn_general_normal_form}Let $G$ be a group, $H,K\le G$ and $\psi:K\to H$ an isomorphism. 
	Choose classes of representatives $T_H$ and $T_K$ of the right cosets for $H$ and $K$ in $G$, respectively, such that $T_H$ and $T_K$ both contain the identity element $e_G$. Then the HNN-extension $G*_\psi$ has a normal form given by the set of words of the form $g_0t^{\varepsilon_1}g_1t^{\varepsilon_2}\ldots t^{\varepsilon_n}g_n$ for $n\ge 0$, $g_0\in G$ and $\varepsilon_i\in \{+1,-1\}$ for $i=1,\ldots,n$, such that
	\begin{enumerate}
		\item $\varepsilon_i=1 $ implies $g_i\in T_K$,
		\item $\varepsilon_i=-1$ implies $g_i\in T_H$ and 
		\item there is no subword of the form $t^\varepsilon e_G t^{-\varepsilon}$, for $\varepsilon\in\{+1,-1\}$.
	\end{enumerate} 
\end{proposition}

An important example of HNN-extensions are \textit{Baumslag-Solitar groups}: given non-zero integers $m,n\in \mathbb{Z}\backslash\{0\}$ we define the group 
$$\mathrm{BS}(m,n)=\langle a,b \left|\right. ba^mb^{-1}=a^n\rangle,$$ 
which is an HNN-extension of $\mathbb{Z}$ via the isomorphism $\psi:\langle a^m \rangle\to \langle a^n\rangle$ defined by $\psi(a^m)=a^n$. When $|m|=1$, these groups form part of a special class of HNN-extensions:
\begin{definition}\label{def:ascending hnn extension} Given a group $G$ and a monomorphism $\psi:G\to G$, we say that $G*_{\psi}$ is an \textit{ascending HNN-extension}.
\end{definition}

Throughout the rest of this work we center our attention in ascending HNN-extensions $G*_{\psi}$ having a finitely generated abelian base group $G$, on which the monomorphism $\psi$ is not surjective. That is, the image $\psi(G)$ is a proper subgroup of $G$. The main example for our results are the Baumslag-Solitar groups $\mathrm{BS}(1,N)$ for $N\ge 2$.

Restricting ourselves to this subfamily of HNN-extensions has many advantages: we have useful commuting relations between the elements of the group, a simpler normal form (in contrast to the general one given by Britton's lemma) and a simpler understanding of the geometric structure of the Cayley graph of the group, all of which we describe in further detail below.

First we establish the following lemma, which summarizes some of the consequences induced by the defining relations used in the standard presentation of $G*_{\psi}$. Its elementary proof is by induction and is omitted.

\begin{lemma}\label{lemma:ascending_hnn_further_identifications}
	Under the above hypotheses and the presentation of $G*_{\psi}$ from Definition \ref{definition.hnn_extension}, for $g\in G$ and $j\ge 0$ we have $	t^jg=\psi^j(g)t^j,$ where $\psi^j$ denotes the composition of $\psi$ with itself $j$ times.
\end{lemma}

\begin{proposition}\label{prop:normalform_ascending_hnn}
	Let $G$ be a finitely generated abelian group, $\psi:G\to G$ a monomorphism. The ascending HNN-extension $G*_{\psi}$ has a normal form given by words of the form $t^{-j}gt^{i}$, where $i,j\ge 0$, $g\in G$, and $g\in \psi(G)$ is possible only if $i=0$ or $j=0$.
\end{proposition}
\begin{proof}
	Using the normal form for HNN-extensions described in Proposition \ref{prop:hnn_general_normal_form} with  $T_K=\{e_{G}\}$ and $T_H$ a set of representatives of $\psi(G)$, we distinguish three types of words.
	\begin{enumerate}[1.]
		\item The first type are words of the form $gt^i$, $g\in G$, $i\ge 0$. This is already a word of the form described in the proposition.
		\item The second type are words of the form $
		g_0t^{-1}g_1\cdots t^{-1}g_n,$ where $ g_0\in G, \ g_1,\ldots,g_n\in T_H.$
		We will prove that we can arrive at a word of the form described in the proposition by induction. Moreover, we will show that there exists $\tilde{g}\in G$ such that
		$$
		g_0t^{-1}g_1\cdots t^{-1}g_n=t^{-n}\psi(\tilde{g})g_n.
		$$\\
		
		If $n=1$, then $g_0t^{-1}g_1=t^{-1}\psi(g_0)g_1$.\\
		
		In the general case, for $n\ge 2$, the induction hypothesis gives us the existence if $\tilde{g}\in G$ such that
		$$
		g_0t^{-1}g_1\cdots t^{-1}g_{n-1}=t^{-(n-1)}\psi(\tilde{g})g_{n-1}.
		$$
		Then
		\begin{align*}
		g_0t^{-1}g_1\cdots t^{-1}g_{n-1}t^{-1}g_n&=t^{-(n-1)}\psi(\tilde{g})g_{n-1}t^{-1}g_n\\
		&=t^{-(n-1)}\psi(\tilde{g})t^{-1}\psi(g_{n-1})g_n\\
		&=t^{-n}\psi(\psi(\tilde{g})g_{n-1})g_n,
		\end{align*}
		finishing the induction.
		\item Finally, the third possible type of words are of the form $
		g_0t^{-1}g_1\cdots t^{-1}g_nt^{j}, \text{ where} \ g_0\in G, \ g_1,\ldots,g_n\in T_H, \ j\ge 1.
		$
		As in the previous case, it is possible to prove that there exists $\tilde{g}\in G$ such that
		$$
		g_0t^{-1}g_1\cdots t^{-1}g_nt^{j}=t^{-n}\psi(\tilde{g})g_nt^j.
		$$
		Note that as $j\ge 1$, the third condition from Proposition \ref{prop:hnn_general_normal_form} ensures that $g_n\neq e_G$ (since otherwise the considered word would contain a subword of the form $t^{-1}e_Gt$), and hence $\psi(\tilde{g})g_n\notin \psi(G)$.
	\end{enumerate}
\end{proof}

An ascending HNN-extension $G*_{\psi}$ whose base group $G$ is a finitely generated abelian group is amenable. We exhibit an explicit F\o lner sequence in the case of Baumslag-Solitar groups, which will be useful in the study of graph-coloring subshifts below. In particular, we will use this sequence to compute topological entropy in $\BS$-subshifts.

\begin{proposition}\label{prop:rectangles_and_folner} Consider the Baumslag-Solitar group $\BS$, $N\ge 2$. Define for $m\ge 1$ the \textit{rectangle of height $m$}
	$$
	R_m\coloneqq \{a^kb^i\mid 0\le k<N^m, \ 0\le i<m  \}\subseteq \BS.
	$$
	Then $\{R_m\}_{m\ge 1}$ forms a F\o lner sequence in $\BS$.
\end{proposition}


\subsection{The Cayley graph }\label{subsection:Cayley}

To finish this section let us take a look at the structure of the Cayley graph of ascending HNN-extensions of finitely generated abelian groups, in particular that of $\BS$, $N\ge 2$. 

Recall that given a group $G$ together with a finite generating set $S\subseteq G$, the \textit{(right) Cayley graph} of $G$ with respect to $S$ is the labeled directed graph  $\Gamma(G,S)=(V_{\Gamma},E_{\Gamma},\lambda_{\Gamma})$ where
\begin{enumerate}
	\item the set of vertices is $V_\Gamma=G$, 
	\item the set of edges is $E_{\Gamma}=\{(g,gs)\in G\times G\mid s\in S\}$, and
	\item the labeling function $\lambda_{\Gamma}:E_{\Gamma}\to V_{\Gamma}$ satisfies $\lambda_{\Gamma}(g,gs)=s$, for $g\in G$, $s\in S$.
\end{enumerate}
For $g\in G$ we denote by $|g|_S$ the minimum distance between $g$ and the identity, in the Cayley graph $\Gamma(G,S)$. When the generating set $S$ is fixed and understood from the context, we simply write $|g|$. We call $|\cdot|_S$ the word metric in $G$ associated with $S$.

A section of the Cayley graph of $\BS$ with respect to the generating set $S=\{a,b\}$ is shown in Figure \ref{fig:section_cayley_graph_bs}, where we see that its structure is that of rows (along the edges labeled by the $a$-generator) being arranged (in a sideways view) as an N-ary tree. Seen from the front, each row has below it (i.e.\ in the $b^{-1}$-direction) a unique $\mathbb{Z}$-section, and above it (i.e.\ in the $b$-direction) $N$ new $\mathbb{Z}$-sections. More formally, defining $A\coloneqq\{a^k \mid k\in\mathbb{Z}\}$, a set of the form $gA$ for $g\in \BS$ will be called an $\mathbf{a}$\textit{-row} of the Cayley graph of $\BS$, meanwhile a set of the form $\displaystyle\bigcup_{n=1}^{\infty}gb^{-n}A  \cup\bigcup_{n=0}^\infty\left( g\prod_{s=1}^{n}(a^{i_s}b)A \right)$ for $g\in \BS$ and a sequence $\{i_s\}_{s=1}^{\infty}\in \{0,\ldots,N-1\}^{\mathbb{N}}$ will be called a \textit{sheet} of the Cayley graph of $\BS$. An example of the latter is illustrated in Figure \ref{fig:bssheet}. 

In the drawings of the Cayley graph of $\BS$ made in the next sections we will omit the label as well as the direction of the edges from the drawing, as they are not relevant for the context they are used in.

The case of the Cayley graph of an ascending HNN-extension $G*_{\psi}$ of a finitely generated abelian group is similar to the one described above: it is composed of copies of the Cayley graph of $G$ connected through the stable letter $t$, in such a way that each one of them has above it $|G:\psi(G)|$ copies of itself. We call these copies of $G$ the $G$\textit{-sections} of the Cayley graph. We analogously define the notion of a \textit{sheet} of the Cayley graph of the HNN-extension $G*_{\psi}$.

We finish this section by making the following definitions, which enable us to enunciate the results of the next section in a clearer way.
\begin{definition}
	For $h\in G$ we define the \textit{$G$-section containing $h$} as
	$$
	\Gamma_{h}\coloneqq \{hg\mid g\in G \}.
	$$
	Using the normal form from Proposition \ref{prop:normalform_ascending_hnn} we can write $h=t^{-j}gt^i$, for $i,j\ge 0$, $g\in G$. In this case we say that $\Gamma_h$ is a \textit{$G$-section of level $i-j$}. Note that thanks to the uniqueness of the normal form, the value $i-j$ is well defined for $h$.
	
	Moreover, note that the $G$-section containing $h$ is naturally in bijection with $G$, via the map $\Gamma_h\to G$ defined by $hg\mapsto g$, for $g\in G$. Hence the restriction of a configuration $x\in \mathcal{A}^{G*_{\psi}}$ can be interpreted as a configuration $\hat{x}\in \mathcal{A}^{G}$, where $\hat{x}_{g}\coloneqq x_{hg}$.
	
	Given a subgroup $H\leqslant G$ and a configuration $x\in \mathcal{A}^{G*_{\psi}}$, we say that \textit{the $G$-section $\Gamma_h$ of $x$ is $H$-periodic} if $\hat{x}$ is $H$-periodic as a configuration in $\mathcal{A}^{G}$, according to the notation of the previous paragraph. That is, if $x_{hgg'}=x_{hg}$,  for every $g\in G$, $g'\in H$.
\end{definition}
\begin{figure}[h]
	\centering
	\includegraphics[scale=0.8]{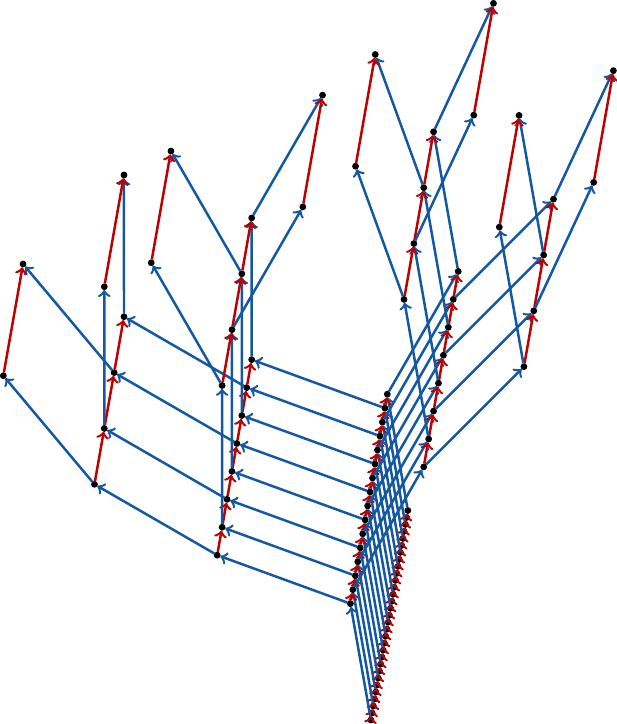}
	\caption{A section of the Cayley graph of $BS(1,2)$}
	\label{fig:section_cayley_graph_bs}
\end{figure}
\begin{figure}[h]
	\centering
	\includegraphics[scale=0.9]{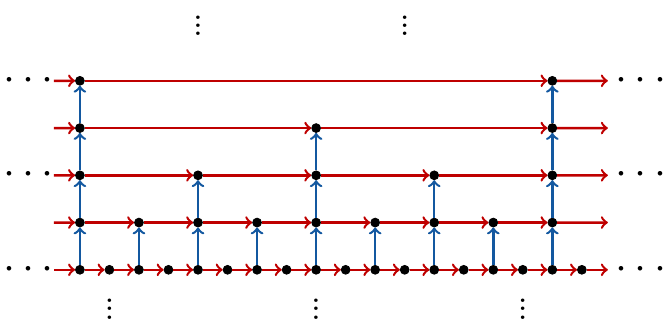}
	\caption{A section of a sheet of the Cayley graph of $BS(1,2)$.  Horizontal edges are labeled with ``$a$'', while vertical edges are labeled with ``$b$''}
	\label{fig:bssheet}
\end{figure}

The next notion allows us to distinguish $G$-sections of the same level, according to what is the highest level of intersection between sheets containing them.

\begin{definition} Consider an ascending HNN-extension $G*_{\psi}$ and a subgroup $H\leqslant G$. For $i,\ell\in \mathbb{Z}$, $\ell < i$, we say that two $G$-sections $\Gamma_h, \Gamma_{h'}$ of the same level $i$ \textit{share a common $G$-section of level $\ell$} if there exists a $G$-section $\Gamma_{h''}$ of level $\ell$ such that $\Gamma_{h}\subseteq \Gamma_{h''}t^{i-\ell}$ and $\Gamma_{h'}\subseteq \Gamma_{h''}t^{i-\ell}$.
	
	Intuitively, this means that starting at the $G$-section $\Gamma_{h''}$ we can reach both $G$-sections $\Gamma_{h}$ and $\Gamma_{h'}$ by only moving in the $t$ direction $i-\ell$ times.
\end{definition}


\section{Periodicity constraints}\label{section:weak_periodicity}

Consider $G$ a finitely generated abelian group and $G*_{\psi}$ an ascending HNN-extension. This section focuses on studying how (weakly) periodic configurations on $\mathcal{A}^{G*_{\psi}}$ exhibit some rigidity when their stabilizer and the structure of the Cayley graph of the group synchronize in some sense. In particular, we will study $G$-periodic configurations and obtain results regarding the structure of their $G$-sections.
\begin{proposition}\label{prop:ascending HNN extension periodicity}
	Suppose there exists $x\in \mathcal{A}^{G*_{\psi}}$, $H\leqslant G$ with $\psi(H)\leqslant H$ and $\ell \ge 0$ such that $\psi^{\ell}(H)\leqslant \mathrm{Stab}(x)$. Then all $G$-sections of $x$ of level greater or equal than $\ell$ are $H$-periodic.
	
	In the case $H=G$, $x$ assigns the same symbol to all $G$-sections of the same level sharing a common $G$-section of level greater or equal than $\ell$.
\end{proposition}
\begin{proof}
	Consider $x\in \mathcal{A}^{G*_{\psi}}$ as above, so that
	\begin{equation}\label{eq:weak H periodicity}
	\sigma_{\psi^{\ell}(h)}(x)=x,\ \text{for all }h\in H.
	\end{equation}
	
	It suffices to prove that for any $i,j\ge 0$, $g\in G$ and $h\in H$ we have that
	\begin{equation}\label{eq: weak H periodicity - equation to prove}
	x_{t^{-j}gt^{i+j+\ell}h}=x_{t^{-j}gt^{i+j+\ell}}.
	\end{equation}
	In effect, note that using Lemma \ref{lemma:ascending_hnn_further_identifications} and commutativity of $G$ we have that
	\begin{align*}
	t^{-j}gt^{i+j+\ell}h&=t^{-j}g\psi^{i+j+\ell}(h)t^{i+j+\ell}	\\
	&=t^{-j}\psi^{i+j+\ell}(h)gt^{i+j+\ell}	\\
	&=\psi^{i+\ell}(h)t^{-j}gt^{i+j+\ell}\\
	&=\psi^{\ell}(\psi^i(h))t^{-j}gt^{i+j+\ell}.
	\end{align*}
	As $\psi^i(h)\in H$ thanks to the hypothesis, Equation \eqref{eq: weak H periodicity - equation to prove} follows from Equation  \eqref{eq:weak H periodicity}.

	For the second part of the proposition, it suffices to prove that for any $i,j,h\ge 0$ and $g,\tilde{g}\in G$ we have
	$$
	x_{t^{-j}gt^{i+j+\ell+h}}=x_{t^{-j}gt^{i+j+\ell}\tilde{g}t^h}.
	$$
	
	Again using commutativity of $G$ and Lemma \ref{lemma:ascending_hnn_further_identifications} we have
	\begin{align*}
	t^{-j}gt^{i+j+\ell}\tilde{g}t^h&=t^{-j}g\psi^{i+j+\ell}(\tilde{g})t^{i+j+\ell+h}\\
	&=t^{-j}\psi^{i+j+\ell}(\tilde{g})gt^{i+j+\ell+h}\\
	&=\psi^{i+\ell}(\tilde{g})t^{-j}gt^{i+j+\ell+h},
	\end{align*}
	and the statement follows from $x$ being $\psi^{\ell}(G)$-periodic.
\end{proof}

\begin{corollary}
	Let $x\in \mathcal{A}^{G*_{\psi}}$ be a strongly periodic configuration, and suppose there exist $H\leqslant G$ with $\psi(H)\leqslant H$ and $\ell \ge 0$ such that $\psi^{\ell}(H)\leqslant \mathrm{Stab}(x)$. Then all $G$-sections of $x$ are $H$-periodic.
\end{corollary}
\begin{proof}
	Since $x$ is strongly periodic, it has a finite orbit under the action of $G*_{\psi}$, so in particular the set $\{ \sigma_{t^{-q}}(x): q\ge 1 \}$ is finite. Hence there exists a sequence of increasing positive integers $\{q_n\}_{n\ge 1}$ such that for every $n\ge1$
	$$
	\sigma_{t^{-q_1}}(x)=\sigma_{t^{-q_n}}(x),
	$$
	or equivalently
	$$
	\sigma_{t^{q_1-q_n}}(x)=x.
	$$
	Fix $g\in G*_{\psi}$, and note that by considering sufficiently large $n$, the element $t^{q_n-q_1}g$ will form part of a $G$-section of level greater or equal than $\ell$. Then using Proposition \ref{prop:ascending HNN extension periodicity} we see that for any $h\in H$:
	\begin{align*}
	x_{g}&=(\sigma_{t^{q_1-q_n}}(x))_{g}\\
	&=x_{t^{q_n-q_1}g}\\
	&=x_{t^{q_n-q_1}gh}\\
	&=\sigma_{t^{q_1-q_n}}(x)_{gh}\\
	&=x_{gh}.
	\end{align*}
	Hence every $G$-section in $x$ is $H$-periodic.
\end{proof}

The behavior observed in the previous theorems is not only related to the configurations appearing on the hypotheses, but also to $G*_{\psi}$-subshifts which contain them. In the following theorem we see how for a $G*_{\psi}$-subshift $X$, having such a weakly periodic configuration as above implies having one with a global structure of $H$-periodic $G$-sections, as well as the existence of strongly periodic configurations with this property for SFTs in the case $H=G$.

\begin{theorem}\label{thm: strong H periodicity HNN extensions in subshifts and SFTs}
	Let $x\in \mathcal{A}^{G*_{\psi}}$, and suppose there exist $H\leqslant G$ with $\psi(H)\leqslant H$ and $\ell \ge 0$ such that $\psi^{\ell}(H)\leqslant \mathrm{Stab}(x)$. Suppose $X$ is a $G*_{\psi}$-subshift containing the configuration $x$. Then there exists a configuration $y\in X$ such that all its $G$-sections are $H$-periodic.
	
	In particular if $H=G$ and $X$ is an SFT, then $X$ contains a strongly periodic configuration for which each of its $G$-sections is monochromatic, that is, it assigns the same symbol to all of their elements, and $G$-sections of the same level share the same symbol.
\end{theorem}
\begin{proof}
	Thanks to Proposition \ref{prop:ascending HNN extension periodicity} we have that all $G$-sections at a level greater or equal than $\ell$ in $x$ are $H$-periodic. We define for each $n\ge 1$ the configuration $x^n\coloneqq\sigma_{t^{-n}}(x)\in X$ and choose $y\in X$ to be a limit point of the sequence $\{x^n\}_{n\in \mathbb{N}}$, which exists by compactness of $X$.

	As each $G$-section at a sufficiently high level of $x$ is $G$-periodic, the construction of the sequence $\{x^n\}_{n\in \mathbb{N}}$ immediately shows that $y$ satisfies the required properties.
	
	Now suppose that $H=G$ and that $X$ is an SFT. As periodicity is preserved through a conjugacy map, we may suppose without loss of generality that $X$ is nearest neighbors SFT. Considering the configuration $y\in X$ from above, by our extra hypothesis of $H=G$ all $G$-sections of $y$ are monochromatic (since they are $G$-periodic), and $G$-sections of the same level are equal.
	
	Using the pigeonhole principle, there must exist a $G$-section of some level $i$ and $h\ge 1$ such that itself together with all $G$-sections of level $i+h$ having the former as a common base share the same symbol. Using the fact that $X$ is a nearest neighbors SFT, we can glue copies of these sections to cover the Cayley graph of $G$ and obtain a strongly periodic configuration $z\in X$.
	
\end{proof}

\begin{remark} Having all $G$-sections monochromatic is not by itself a sufficient condition for a configuration $x\in \mathcal{A}^{G*_{\psi}}$ to be strongly periodic. For example, in the case $G*_{\psi}=\BS$ we can construct an example of such a configuration by considering a non-periodic sequence $z\in \mathcal{A}^\mathbb{Z}$ and define $x\in x\in \mathcal{A}^{\BS}$ to consist of monochromatic $\mathbb{Z}$-sections whose symbols are defined according to the sequence $z$. Figure \ref{fig:monochromatic_rows_tree_counterexample} shows a sideways view of the Cayley graph of $\BS$, where each symbol in this figure represents the symbol of the entire $\mathbb{Z}$-section at that level in $x$.
	\begin{figure}
		\centering
		\includegraphics{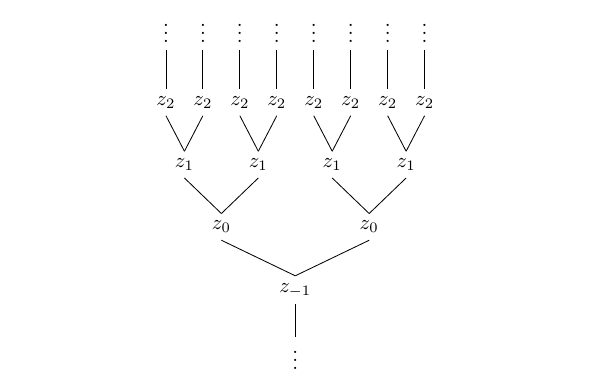}
		\caption{Illustration of a configuration in a sideways view of the Cayley graph of $\BS$, having all $\mathbb{Z}$-sections monochromatic but not being strongly periodic}
		\label{fig:monochromatic_rows_tree_counterexample}
	\end{figure}
\end{remark}


\section{Graph-coloring subshifts}\label{section:graph_coloring_subshifts}
Given a group $G$ generated by a finite subset $S\subseteq G$, we define for $n\ge 2$ the \textit{graph-coloring subshift} (GCS)
$$
\mathcal{C}_{n}\coloneqq\left\{x\in \{0,\ldots,n-1\}^G\mid \text{ for all } g\in G, s\in S: \ x_{g}\neq x_{gs} \right\}.
$$
That is, each configuration $x\in \mathcal{C}_{n}$ describes a proper $n$-coloring of the Cayley graph $\Gamma(G,S)$ of $G$ with respect to the generating set $S$. In what follows we consider the GCS in $n\ge 2$ colors for $\BS$ generated by the standard set $S=\{a,b\}$.

In this section we study some dynamical properties of graph-coloring subshifts on $\BS$, namely mixing conditions and estimates of their entropy, among others. We restrict ourselves to the case of Baumslag-Solitar groups rather than working in the general setting of ascending HNN-extensions of finitely generated abelian groups, since the linear structure of the base group allows for rather simple constructions and results.
\begin{proposition}\label{prop:GCS_nonemptiness} If $N$ is odd $\mathcal{C}_2\neq \varnothing$, and if $N$ is even then $\mathcal{C}_2=\varnothing$ meanwhile $\mathcal{C}_3\neq \varnothing$. In particular, independent of the parity of $N$, for every $n\ge 3$ we have $\mathcal{C}_n\neq \varnothing$.
\end{proposition}
\begin{proof}
	Let us see first that if $N$ is even then $\mathcal{C}_2=\varnothing$. To see this suppose we have $x\in \mathcal{C}_2$ and without loss of generality let us suppose that $x_{e_{\BS}}=0$. Then by the coloring rules we must have $x_{a^{N}}=0$, and $x_{b}=x_{a^Nb}=1$, but this cannot be since $x_{a^Nb}=x_{ba}$, so the neighboring vertices $b$ and $ba$ have the same color and this contradicts the GCS's definition.

	Now let us show that if $N$ is odd then $\mathcal{C}_2\neq \varnothing$. Moreover, we will show that in fact $|\mathcal{C}_2|=2$. To create a point $x\in \mathcal{C}_2$ let us impose that $x_{e_{\BS}}=0$, and define for every $g\in \BS$ expressed in its normal form $g=b^{-j}a^kb^i$ with $i,j\ge 0$ and $k\in \mathbb{Z}$: $x_g=i+j+k \mod 2$. We check that this provides a consistent coloring of the Cayley graph:
	for $g$ as above and a generator $s\in \{a,b\}$, the normal form of $gs$ is given by
	$$
	ga=b^{-j}a^{k+N^i}b^{i},
	$$
	if $s=a$ and 
	$$
	gb=b^{-j}a^{k}b^{i+1},
	$$
	if $s=b$.
	Then $x_{ga}=i+j+k+N^i \mod 2$ and $x_{gb}=i+1+j+k \mod 2$, and as $N$ is odd we see that $x_{ga}\neq x_g$ and $x_{gb}\neq x_g$. With this, neighboring elements in the graph have different colors and hence $x$ forms a valid configuration in $\mathcal{C}_2$. Also note that this point is completely determined by our choice of $x_{e_{\BS}}=0.$ If we had chosen instead $x_{e_{\BS}}=1$, we would have obtained the same configuration as above, but with $0$'s and $1$'s interchanged, thus we conclude that $|\mathcal{C}_2|=2$.

	Let us see now that for $N$ even we have $\mathcal{C}_3\neq\varnothing$. For this notice first that if we have an $\mathbb{Z}$-section $\Gamma_g\coloneqq\{ga^k: k\in \mathbb{Z}\}, g\in G$ colored consistently with the coloring rules using two colors, i.e. the two colors strictly alternate along $\Gamma_g$, then we can color each row directly above it respecting the coloring rules: as $\Gamma_g$ is colored with two colors and $N$ is even, then every row directly above it only sees one color through its edges connecting it to $\Gamma_g$, so if we choose the two remaining colors we can color this new row consistently. On the other hand note that if $\Gamma_g$ is colored consistently, we can also color the row exactly below it: as $\Gamma_g$ is colored with two colors we can color the vertices below it with the remaining color in $\{0,1,2\}$, and as $N$ is even we can choose any other color and fill the gaps consistently between the already colored vertices on this new row, alternating those two colors.

	With the two previous facts it is easy to see that one can construct inductively a point in $\mathcal{C}_3$. Hence $\mathcal{C}_3\neq \varnothing$.

	The final statement of the proposition follows from the fact that if $n> 3$, then $\mathcal{C}_n$ contains a copy of $\mathcal{C}_{n-1}$ by simply considering colorings with a smaller palette of colors.
\end{proof}


\subsection{Extensibility of patterns}\label{subsection:extensibility_patterns}
In the case of $2$-colorings of $\BS$ the subshift $\mathcal{C}_2$ is finite (it is either empty or has two configurations, if $N$ is even or odd, respectively) so the configurations appearing in it are somewhat rigid and admit very few different ways of coloring the Cayley graph. This leads us to ask how many ways to color are there in the case of $n\ge 3$, and with it be able to estimate topological entropy $\htop(\mathcal{C}_n)$.

In order to do this, we will first tackle the question of extending colorings of a finite subset of the Cayley graph to bigger patterns containing it. In the language of symbolic dynamics this question is the same as asking if in $\mathcal{C}_n$ locally admissible patterns are globally admissible, or under which conditions on the pattern we can ensure this. A first result in this direction answers the question affirmatively, subject to having a minimum of available colors.
\begin{proposition} \label{prop:loc_adm_5_colors}
	For $n\ge 5$, every locally admissible pattern $p$ of $\mathcal{C}_n$ is globally admissible.
\end{proposition}
\begin{proof}
	This follows from the fact that for a finite graph, its chromatic number is less than its maximum degree. Since every finite subgraph of the Cayley graph of $\BS$ has maximum degree at most $4$, then the considered pattern can be extended to any finite graph containing it by choosing for each vertex a color that none of its neighbors has, and thus preserving the coloring rules. 
	
	Fix a sequence of finite subsets $\{S_n\}_{n\ge 1}$ of $\BS$, such that
	\begin{itemize}
		\item $S_0=S$,
		\item $S_n\subseteq S_{n+1}$, for $n\ge 0$, and
		\item $\BS=\bigcup_{n\ge 0}S_n$.
	\end{itemize}
	
	Given any locally admissible pattern $p$ with support $S$, we can consider a sequence of configurations $\{x^{n}\}_{n\in \mathbb{N}}$ in $\{0,\ldots,n-1\}^{\BS}$ such that $x^{n}|_{S_n}$ is a proper coloring of $S_n$ and $x^{n}|_{S}=p$, thanks to the above explanation. Now by taking $x\in \{0,\ldots,n-1\}^{BS}$ to be an accumulation point of said sequence, we see that by construction it is a proper coloring in $\mathcal{C}_n$ which coincides with $p$ in $S$.
\end{proof}
\begin{remark} The previous proposition does not hold for $n\in \{3,4\}$. For example the locally admissible pattern in $\BS[2]$ in Figure \ref{fig:no_extend_gap} cannot be realized within a point $x \in \mathcal{C}_4$ since no color can be assigned to the vertex with an interrogation sign, while respecting the coloring rules.
	\begin{figure}
		\centering
		\includegraphics{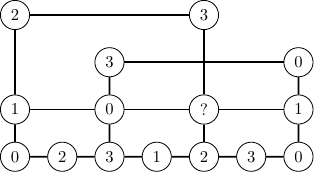}
		\caption{Locally admissible pattern for $\mathcal{C}_4$ which cannot be extended to a global configuration}
		\label{fig:no_extend_gap}
	\end{figure}	
\end{remark}
The previous remark showed a locally admissible pattern that could not be extended to form a valid configuration of the group, and the main property of this example is that although the support of the chosen pattern can be taken to be a connected subgraph of the Cayley graph, it has a ``gap" which allows us to surround a position in such a way that the pattern could not be extended. A way of avoiding these pathological cases is to require the supports of the patterns to have some sort of convexity.
\begin{lemma} \label{lem:gcs_extend_to_rectangles} For every $n\ge 3$ every $n$-coloring of the rectangle $R_m$ can be extended to a proper $n$-coloring of the rectangle $R_{m+1}$.
\end{lemma}
\begin{proof}
	For every $k\in \{0,\ldots,m\}$ we define $$H_k\coloneqq \left(R_{m+1}\backslash R_{m}\right)\cap \{a^ib^k\mid i\ge 0\},$$
	which represents the elements of $R_{m+1}$ of height $k$, outside of the rectangle $R_m$. Note that by definition we have $$R_{m+1}=R_m\cup \bigcup_{k=0}^{m}H_k.$$ Then we can extend the pattern $p$ on $R_m$ by successive colorings (respecting the condition of being a proper coloring) coloring first $H_0$, then $H_1$, and so until we color $H_m$. In each step of this process each vertex has at most $2$ neighbors already colored, so having at least $3$ available colors is enough for this process to be carried out. 
\end{proof}
By using the same ideas present in the proof of the previous lemma we obtain the following proposition. Indeed, one can inductively color the rows of the Cayley graph one by one, so that each vertex has at most two colored neighbors at the moment it must choose its color. 
\begin{proposition}\label{prop:gcs_rectangle_extension} For $n\ge 3$ every locally admissible pattern for $\mathcal{C}_n$ with support a rectangle is globally admissible.
\end{proposition}
\begin{proof}
	Suppose that we have a locally admissible pattern $p$ defined on the rectangle $R_{m}$. First note that we can extend this to a proper coloring on the $0$ level. That is, on the subgroup generated by $a$. In effect, the elements at $a^{-1}$ and at $a^{N^m}$ only have one colored neighbor, so we can assign them a color without introducing forbidden patterns. By the same argument, we can color inductively neighboring vertices until we have a proper coloring of $R_m\cup \langle a \rangle$.
	
	Now we can do the same with upper levels: consider rows of level $1$ which are connected to $\langle a\rangle$. The elements neighboring $R_m$ have at most two colored neighbors, and so we can extend the pattern. Again proceeding inductively and coloring in order according to their relative distance to $R_m$, we see that we can extend the pattern to $R_m\cup \langle a\rangle \cup \langle a\rangle b$. Moreover, by repeating the same argument we can see that it is possible to extend the coloring of $R_m$ to all first $m$ levels. That is, to the set $\langle a\rangle \cup \langle a\rangle b \cup \cdots \cup \langle a\rangle b^{m-1} .$
	
	We can further extend this coloring upwards:color upper levels inductively, one level at a time, and coloring neighboring vertices of already colored vertices at each step. In this way, it is possible to guarantee that each vertex will have at most $2$ already colored neighbors. With this, we see that it is possible to extend the coloring to all elements of the form $a^kb^i$, $k\in \mathbb{Z}$, $i\ge 0$.
	
	Finally, we note that the same method works for extending the coloring into the $b^{-1}$ direction, and into the new sheets that appear when traversing this direction. Again, simply color one level at a time making sure that when coloring one vertex it has at most $2$ already colored neighbors.
\end{proof}


\subsection{Mixing and frozen colorings}\label{subsection:mixing_and frozen}
Thanks to Proposition \ref{prop:loc_adm_5_colors} we know that for $n\ge 5$ we can extend every admissible pattern to a proper coloring of the Cayley graph of $\BS$, from which we obtain the following mixing property for $\mathcal{C}_n$.
\begin{definition} We say that a subshift $X\subseteq \mathcal{A}^G$ is \textit{strongly irreducible} if there exists $F\subseteq G$ finite such that for every globally admissible patterns $p,q$ in $X$ with $\mathrm{supp}(p)\cap \mathrm{supp}(q)\cdot F=\varnothing$: $[p]\cap [q]\neq \varnothing$.
\end{definition}

\begin{theorem}\label{thm:gcs_mixing_n5} Let $n\ge 5$ and consider two non-adjacent finite subsets $F_1,F_2\subseteq \BS$, that is, such that 
	$$\left( F_1\cdot \{e_{\BS},a,a^{-1},b,b^{-1}\}\right)\cap F_2=\varnothing.$$
	Then for every choice of (locally) admissible patterns $p_i\in \{0,\ldots,n-1\}^{F_i}, i=1,2,$ there exists $x\in \mathcal{C}_n$ such that $x|_{F_1}=p_1$ and $x|_{F_2}=p_2$. With this we have that for $n\ge 5$ the GCS $\mathcal{C}_n$ is strongly irreducible.
\end{theorem}
\begin{proof}
	The above condition on $F_1$ and $F_2$ means that these two sets are disjoint, and separated by at least one vertex in the Cayley graph of $\BS$. 
	
	Consider two locally admissible patterns $p_i\in \{0,\ldots,n-1\}^{F_i}, i=1,2$. Then we can define a new pattern $p\in \{0,\ldots,n-1\}^{F_1\cup F_2}$ such that $p(f)=p_i(f)$ for $f\in F_i, i=1,2$ consistently, since $F_1\cap F_2=\varnothing.$ This pattern is also locally admissible thanks to the distance existing between $F_1$ and $F_2$, as between every vertex of $F_1$ and every vertex of $F_2$ there is at least one uncolored vertex. Then, according to Proposition \ref{prop:loc_adm_5_colors}, this pattern is globally admissible and hence there exists $x\in \mathcal{C}_n$ such that $x|_{F_1\cup F_2}=p$. This $x$ satisfies $x|_{F_1}=p_1$ and $x|_{F_2}=p_2$ and so we have found the desired point.
\end{proof}

The previous proposition tells us that for $n\ge 5$ the subshift $\mathcal{C}_n$ has a strong mixing property. On the other side, we can see that $\mathcal{C}_2$ has no type of mixing behavior: this $\BS$-subshift is either empty if $N$ is even, or if $N$ is odd then there is no way to assign the same color to two elements of the group $g,h\in \BS$ such that $g^{-1}h\in \left\{a^{2m+1}\mid m\in \mathbb{Z} \right\}$, which forbids the gluing of patterns at arbitrarily large distances. This contrast between $\mathcal{C}_2$ and $\mathcal{C}_n$ for $n\ge 5$ raises the question of what kind of mixing behaviors does $\mathcal{C}_3$ and $\mathcal{C}_4$ have. To study this we will use a similar approach as that of \cite{alon2019mixing}, by introducing the concept of a frozen coloring.
\begin{definition} Let $n\ge 3$. A configuration $x\in \mathcal{C}_n$ is called a \textit{frozen coloring} if for every $y\in \mathcal{C}_n$ such that there exists a non-empty finite subset $F\subseteq \BS$ with $x|_{F^c}=y|_{F^c}$, then $x= y$. That is, no coloring of $\BS$ other than $x$ can coincide with it outside of a finite set.
\end{definition}

Frozen colorings are configurations in which the neighboring vertices of every finite subset of the Cayley graph determine unequivocally how this subset must be colored, in order to obtain a proper coloring. This behavior is the reason frozen colorings are closely related to the (lack of) mixing properties of the GCS, as remarked in Theorem 4.4 of \cite{alon2019mixing} for $\mathbb{Z}^d$. In the following proposition we state and prove the precise result in our context.
\begin{proposition} \label{prop:froz_col_not_si} If $\mathcal{C}_n$ has a frozen coloring and $|\mathcal{C}_n|\ge 2$, then it is not strongly irreducible.
\end{proposition}
\begin{proof}
	
	Looking for a contradiction, suppose that $\mathcal{C}_n$ is strongly irreducible and $x\in \mathcal{C}_n$ is a frozen coloring. 
	
	Consider any other configuration $y\in \mathcal{C}_n$ such that $y_{e_{\BS}}\neq x_{e_{\BS}}$. As $\mathcal{C}_n$ is strongly irreducible, there exists $F\subseteq G$ finite such that for any two patterns $p,q\in \mathcal{L}(\mathcal{C}_n)$ with $\mathrm{supp}(p)\cap \mathrm{supp}(q)\cdot F=\varnothing$ we have $[p]\cap [q]\neq \varnothing$. Denote by $|\cdot|$ the word metric in $\BS$, as defined in Subsection \ref{subsection:Cayley}. Considering the patterns $y|_{e_{\BS}}$ and $x|_{\partial B_M}$, where $\partial B_M\coloneqq \{g\in \BS: |g|=M\}$, for $M$ sufficiently large we have $\partial B_M \cap F= \varnothing$. Then there must exist a coloring $z\in \mathcal{C}_n$ such that $z_{e_{\BS}}=y_{e_{\BS}}\neq x_{e_{\BS}}$, and $z|_{\partial B_M}=x|_{\partial B_M}$. Moreover, we can assume that $z|_{B_{M-1}^c}=x|_{B_{M-1}^c}$, where $B_{M-1}\coloneqq \{g\in \BS: |g|\le M-1 \}$, since $z$ and $x$ coincide on $\partial B_M$ and hence re-coloring $z$ as $x$ outside of this ball gives us a proper coloring. 
	
	With this we have found a configuration $z\in \mathcal{C}_n$ which coincides with $x$ outside of a finite set but is different from $x$ inside it, so we have a contradiction with the fact that $x$ is a frozen coloring. Hence we conclude that a strongly irreducible subshift $\mathcal{C}_n$ cannot have a frozen coloring.
\end{proof}
We already know by Theorem \ref{thm:gcs_mixing_n5} that the GCS $\mathcal{C}_n$ for $n\ge 5$ is strongly irreducible, and hence by the above proposition it cannot have a frozen coloring. On the other hand, as the case $n=3$ is the first non-finite subshift of the GCS's $\mathcal{C}_n$ for $N\ge 2$, it is reasonable to conjecture that its properties must still be somewhat rigid. The next proposition confirms this by showing that $\mathcal{C}_3$ possesses a frozen coloring, and with it exhibits the lack of a strong mixing behavior of the GCS with three colors.
\begin{theorem}\label{thm:existence_frozen_3} The GCS $\mathcal{C}_3\subseteq \{0,1,2\}^{\BS}$ admits a frozen coloring, and hence is not strongly irreducible.
\end{theorem}
\begin{proof}
	The proof will be divided in three cases, according to the value of $N\ (\mathrm{mod} \ 3)$. In each case we will define explicitly a frozen coloring, by defining the value of each element of $\mathrm{BS}(1,N)$ according to its normal form.

	Let us suppose first that $N=1 \ (\mathrm{mod} \ 3)$, and with it that for every $i\ge 0$ we have $N^i=1 \ (\mathrm{mod} \ 3)$. Define the configuration $x\in \{0,1,2\}^{\BS}$ in the following way. For $g=b^{-j}a^kb^i\in \BS$ written in its normal form:
	
	$$
	x_g\coloneqq 2(i-j)+k \ (\mathrm{mod} \ 3).
	$$
	This configuration is illustrated in Figure \ref{fig:gcs_frozen_3_col_N_1} for $N=4$.
	\begin{figure}
		\centering
		\includegraphics[scale=0.6]{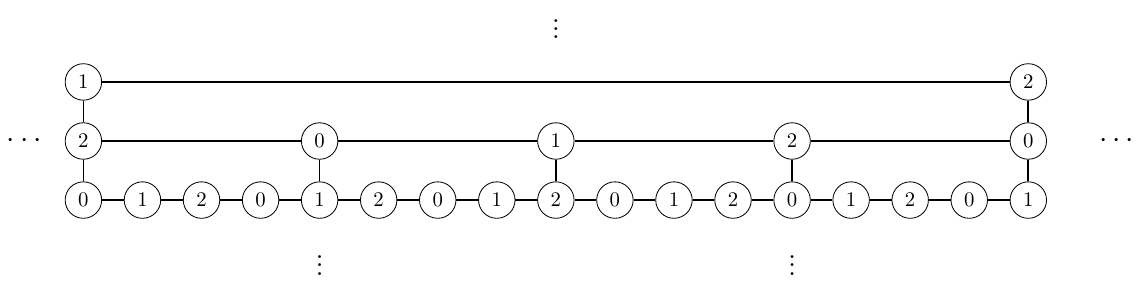}
		\caption{Construction of a frozen $3$-coloring for $N=1\ (\mathrm{mod} \ 3)$. Here $N=4$.}
		\label{fig:gcs_frozen_3_col_N_1}
	\end{figure}
	
	Note that $gb=b^{-j}a^kb^{i+1}$ and that $ga=b^{-j}a^kb^{i}a=b^{-j}a^{k+N^{i}}b^{i}$. Hence, using the fact that $N^i=1 \ (\mathrm{mod} \ 3)$, we see that
	\begin{align*}
	x_{gb}&= x_g+2 \ (\mathrm{mod} \ 3) \text{ and }\\
	x_{ga}&=x_{g}+N^{i} \ (\mathrm{mod} \ 3)=x_{g}+1 \ (\mathrm{mod} \ 3).
	\end{align*}

	Therefore $x_{gb}\neq x_g$ and $x_{ga}\neq x_g$, and so $x\in \mathcal{C}_3$ defines a proper coloring.

	Let us see that $x$ defines a frozen coloring: looking for a contradiction let us suppose $y\in \mathcal{C}_3$ is such that there exists a finite subset $F\subseteq \BS$ with $y|_{F^ c}=x|_{F^c}$ and for every $f\in F: \ x_f\neq y_f$.
	
	Now consider $g$ in $F$ which maximizes the value of $i+k$ in its normal form. In other words,
	$$
	g\coloneqq \mathrm{argmax}\{i+k\mid g=b^{-j}a^kb^i\in F, \ k\in \mathbb{Z}, i,j\ge 0 \}.
	$$
	By definition, we must have that $gb\notin F$ and that $ga\notin F$, since otherwise we would arrive at a contradiction with the maximality of the coordinates of $g$.
	
	As $y$ and $x$ coincide outside of $F$, we see that $y_{gb}=x_{gb}=x_g+2\ (\mathrm{mod} \ 3)$, and $y_{ga}=x_{ga}=x_g+1\ (\mathrm{mod} \ 3)$, Moreover, as $y$ defines a proper coloring we must have $y_g=x_g$. This gives a contradiction since as $g\in F$ we should have $y_g\neq x_g$.
	
	Now suppose that $N=2\ (\mathrm{mod} \ 3)$. Note that now we have that, for $i\ge 0$,
	\begin{equation*}	
	N^i=\left\{ 
	\begin{aligned}
	&1 \ (\mathrm{mod} \ 3) \text{ if }i\text{ is even}, \\
	&2 \ (\mathrm{mod} \ 3) \text{ if }i\text{ is odd.} 
	\end{aligned}
	\right.
	\end{equation*}
	Define a configuration  $x\in \{0,1,2\}^{\BS}$ such that for $g=b^{-j}a^k b^i\in \BS$ written in its normal form:
	$$
	x_g\coloneqq i-j+k \ (\mathrm{mod} \ 3).
	$$
	This configuration is illustrated in Figure \ref{fig:gcs_frozen_3_col_N_2}.
	\begin{figure}
		\centering
		\includegraphics[scale=0.6]{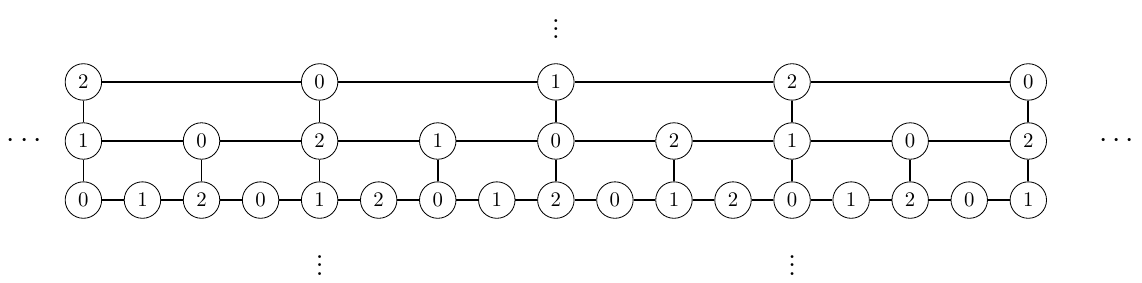}
		\caption{Construction of a frozen $3$-coloring for $N=2\ (\mathrm{mod} \ 3)$. Here $N=2$.}
		\label{fig:gcs_frozen_3_col_N_2}
	\end{figure}	
	
	Then, similar as before, we have
	\begin{align*}
	x_{gb}&=x_g+1 \ (\mathrm{mod} \ 3), \text{ and} \\
	x_{ga}&=x_g+N^i \ (\mathrm{mod} \ 3),
	\end{align*}
	and so $x_{gb}\neq x_g$ and $x_{ga}\neq x_g$ (since $N^i$ can be either $1$ or $2$). With this we have that $x\in \mathcal{C}_3$ defines a proper coloring, so it only remains to prove that $x$ defines a frozen coloring. Looking for a contradiction, suppose there exists $y\in \mathcal{C}_3$ and a finite set $F\subseteq \BS$ such that $x|_{F^c}=y|_{F^c}$ and for every $f\in F: x_f\neq y_f$.
	Define $i_{\mathrm{max}}\coloneqq \max\{ i\ge 0\mid b^{-j}a^kb^i\in F, \text{ for some }k\in \mathbb{Z}, \ j\ge 0 \}$ and $B_{\mathrm{max}}\coloneqq\{g=b^{-j}a^kb^{i_{\mathrm{max}}}\mid\  g\in F\}$.
	
	If $i_{\mathrm{max}}$ is even, let us take $g=b^{-j}a^kb^i\in B_{max}$ which minimizes the value of $k$. Then as $i_{max}$ is even we have $N^{i_{max}}=1  \ (\mathrm{mod} \ 3)$ and with it $y_{gb}=x_{gb}=x_g+1  \ (\mathrm{mod} \ 3)$, and $y_{ga^{-1}}=x_{ga^{-1}}=x_g-N^{i_{max}} = x_g+2  \ (\mathrm{mod} \ 3)$. But then we must have $y_g=x_g$ and this gives a contradiction since $g\in F$ and therefore $x_g\neq y_g$.
	
	Now if $i_{\mathrm{max}}$ is odd, let us consider the element $g=b^{-j}a^kb^i\in B_{max}$ which maximizes the value of $k$. Then as $i_{\mathrm{max}}$ is odd we have $N^{i_{\mathrm{max}}}=2  \ (\mathrm{mod} \ 3)$ and with it $y_{gb}=x_{gb}=x_g+1  \ (\mathrm{mod} \ 3)$, and $y_{ga}=x_{ga}=x_g+N^{i_{\mathrm{max}}} = x_g+2  \ (\mathrm{mod} \ 3)$. But then we must have $y_g=x_g$ and this gives a contradiction since $g\in F$ and therefore $x_g\neq y_g$, proving that $x$ defines a frozen coloring.

	Finally suppose that $N=0 \ (\mathrm{mod} \ 3)$. The method used on the previous two cases to find a frozen coloring on $\mathcal{C}_3$ was to construct a configuration using a function $f(j,k,i)$ of the coefficients of the normal form of every element of the group, which had period $3$ on the variable $k$. This cannot be done in the case $N\in 3\mathbb{Z }$ since here the configuration constructed would satisfy $x=\sigma_{a^3}(x)$ and hence $x=\sigma_{a^N}(x)$ (since $N\in 3\mathbb{Z}$). But then by Proposition \ref{prop:ascending HNN extension periodicity} the coloring $x$ would have to have monochromatic rows appearing on it, which contradicts the fact that $x\in \mathcal{C}_3$ defines a proper coloring. Nonetheless, using a similar but different kind of function we can still manage to prove the existence of a frozen coloring.
	
	In order to make the proof more clear, we will start with the case $N=3$. Define the configuration $x\in \{0,1,2\}^{\BS}$ such that for $g=b^{-j}a^k b^i \in \BS$ written in its normal form:
	$$
	x_g=(k \ \mathrm{mod}\ 2)+ 2(i-j)\ \mathrm{mod}\ 3.
	$$
	
	This configuration is illustrated in Figure \ref{fig:gcs_frozen_3_col_N_0}.
	
	\begin{figure}
		\centering
		\includegraphics[scale=0.6]{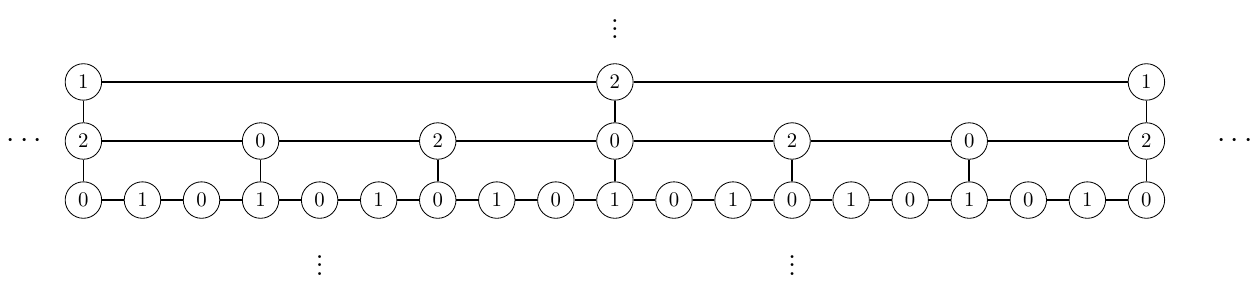}
		\caption{Construction of a frozen $3$-coloring for $N=0\ (\mathrm{mod} \ 3)$. Here $N=3$.}
		\label{fig:gcs_frozen_3_col_N_0}
	\end{figure}

	Then, using that $gb=b^{-j}a^kb^{i+1}$ and $ga=b^{-j}a^{k+3^i}b^{i}$, we have 
	\begin{equation*}
	x_{gb}= (k\ \mathrm{mod}\ 2)+2(i+1-j)\ \mathrm{mod}\ 3=x_{g}+2\  \mathrm{mod} \ 3, \text{ and }
	\end{equation*}
	\begin{align*}
	x_{ga}&=(k+3^{i}\ \mathrm{mod} \ 2)+ 2(i-j) \ \mathrm{mod}\ 3 \\
	&=(k+1\ \mathrm{mod} \ 2)+ 2(i-j) \ \mathrm{mod}\ 3\\
	&=\left\{\begin{aligned}
	x_{g}+1 \ \mathrm{mod}\ 3, &\text{ if }k=0\ \mathrm{mod}\ 2, \\
	x_{g}+2 \ \mathrm{mod}\ 3, &\text{ if }k=1\ \mathrm{mod}\ 2.
	\end{aligned} \right.
	\end{align*}
	Again, in all cases $x_{gb}$ and $x_{ga}$ are different from $x$ and so $x\in \mathcal{C}_3$ defines a proper coloring.
	
	Let us now see that $x$ is a frozen coloring. Suppose that there exists $y\in \mathcal{C}_3$ and $F\subseteq \mathrm{BS}(1,3)$ finite with $x|_{F^{c}}=y|_{F^{c}}$, and such that for any $f\in F$: $x_f\neq y_f$.
	
	Define $k_{\mathrm{max}}\coloneqq\{ k\in \mathbb{Z}\mid \ b^{-j}a^kb^i\in F \}$. We have two cases. If $k_\mathrm{max}=0\ \mathrm{mod}\ 2$, choose $g=b^{-j}a^{k_{\mathrm{max}}}b^i\in F$ which maximizes the value of $i$. Then $gb\notin F$ and $ga\notin F$. But the by definition of $x$ and as $y$ coincides with $x$ outside $F$, we must have
	\begin{equation*}
	y_{gb}=x_{gb}=x_{g}+2\ \mathrm{mod}\ 3, \text{ and  } 
	y_{ga}=x_{ga}=x_{g}+1\ \mathrm{mod}\ 3.
	\end{equation*}
	This is a contradiction, since it implies that $y_g=x_g$, for $g\in F$.

	The case $k_{\mathrm{max}}=1\ \mathrm{mod}\ 2$ is similar. Choose $g=b^{-j}a^{k_\mathrm{max}}b^{i}\in F$ which minimizes $(i-j)$. Then $ga\not \in F$ and so $y_{ga}=x_{ga}=x_{g}+2\ \mathrm{mod}\ 3$. On the other hand, we also have that $gb\not \in F$.
	If $i\ge 1$ this implies that
	\begin{align*}
	x_{gb^{-1}}&=(k\ \mathrm{mod\ 2})+2(i-j-1)\ \mathrm{mod}\ 3\\
	&=(k\ \mathrm{mod\  2})+2(i-j)+1\ \mathrm{mod}\ 3\\
	&=x_g+1\ \mathrm{mod}\ 3.
	\end{align*}
	Finally, if $i=0$, then 
	\begin{align*}
	x_{gb^{-1}}&=x_{b^{-j-1}}a^{k_{\mathrm{max}}\cdot 3}\\
	&=(k_{\mathrm{max}}\cdot 3 \ \mathrm{mod}\ 2)+2(-j-1)\ \mathrm{mod}\ 3\\
	&=x_g+1\ \mathrm{mod}\ 3.
	\end{align*}
	
	In all cases, $y_{gb^{-1}}=x_{gb^{-1}}=x_{g}+1\ \mathrm{mod} \ 3$. As we had already seen that $y_{ga}=x_{g}+2\ \mathrm{mod}\ 3$, this implies that $y_g=x_g$, which is a contradiction since $g\in F$.

	Now we can do the general case, that is, suppose that $N=0 \ \mathrm{mod} \ 3$, for $N \ge 6$. Define the configuration $x\in \{0,1,2\}^{\BS}$ by assigning to each $g=b^{-j}a^kb^{i}\in \BS$ written in its normal form:
	$$
	x_g=(k\ \mathrm{mod}\ (N-1))+ 2(i-j)\ \mathrm{mod}\  3.
	$$ 
	
	Then $x_{gb}=x_g+2 \ \mathrm{mod} \ 3$, and
	\begin{align*}
	x_{ga}&=(k+N^i \ \mathrm{mod}\ (N-1))+2(i-j)\ \mathrm{mod}\ 3\\
	&=(k+1 \ \mathrm{mod}\ (N-1))+2(i-j)\ \mathrm{mod}\ 3\\
	&=\left\{
	\begin{aligned}
	x_g+1 \ \mathrm{mod}\ 3 &\text{ if } 0\le k\ \mathrm{mod}\ (N-1) <N-2\\
	x_g-(N-2) \ \mathrm{mod}\ 3 &\text{ if }  k=N-2 \ \mathrm{mod}\ (N-1).
	\end{aligned}
	\right.
	\end{align*}	
	
	Hence, using that $N=0\ \mathrm{mod}\ 3$, we have that
	$$
	x_{ga}=\begin{cases}
	x_g+1 \ \mathrm{mod} \ 3&\text{ if } 0\le k\ \mathrm{mod}\ (N-1) <N-2\\
	x_{g}+2 \ \mathrm{mod} \ 3, &\text{ if }k=N-2 \ \mathrm{mod}\ (N-1).
	\end{cases}
	$$	
	We conclude that $x\in \mathcal{C}_3$ defines a proper coloring.
	
	Now let us prove that $x$ is a frozen coloring. Suppose there exists $y\in \mathcal{C}_3$ and a finite subset $F\subseteq\BS$ with $x|_{F^{c}}=y|_{F^c}$ and such that for every $f\in F$: $x_f\neq y_f$.
	
	Define $k_{\mathrm{max}}\coloneqq \max\{k\in \mathbb{Z}\mid \ b^{-j}a^kb^i\in F \}$. We have two cases.
	
	First, suppose that $k_{\mathrm{max}} \ \mathrm{mod} (N-1)<N-2$. Then choose an element $g=b^{-j}a^kb^i\in F$ which maximizes $i$, so that $gb\notin F$ and $ga\notin F$. We conclude that $y_{gb}=x_{gb}=x_g+2\ \mathrm{mod}\ 3$ and that $y_{ga}=x_{ga}=x_g+1\ \mathrm{mod}\ 3$, which leads to the  equality $x_g=y_g$. This is a contradiction with the choice of $y$.
	
	Now suppose that $k_{\mathrm{max}}=N-2\ \mathrm{mod}\ (N-1)$. Then choose $g=b^{-j}a^{k_{\mathrm{max}}}b^i\in F$ which minimizes the value of $(i-j)$. As $ga\notin F$, we have that $y_{ga}=x_{ga}=x_g+2\ \mathrm{mod}\ 3$. On the other hand, by the choice of $g$, we have that $gb^{-1}\notin F$.
	
	If $i\ge 1$, then $gb^{-1}=b^{-j}a^kb^{i-1}$ and so
	\begin{align*}
	y_{gb^{-1}}=x_{gb^{-1}}&=(k\ \mathrm{mod}\ (N-1))+2(i-1-j)\ \mathrm{mod}\ 3\\
	&=x_g+1\ \mathrm{mod}\ 3.
	\end{align*}
	Meanwhile, if $i=0$ then $gb^{-1}=b^{-j}a^kb^{-1}=b^{-j-1}a^{kN}$, and so
	$$
	y_{gb^{-1}}=x_{gb^{-1}}=(kN \ \mathrm{mod}\ (N-1))+2(-j-1)\ \mathrm{mod}\ 3=x_{g}+1\ \mathrm{mod}\ 3.
	$$
	In both cases, we conclude that $y_g=x_g$ and this is a contradiction.

	To finish the proof we simply use the previous proposition to see that $\mathcal{C}_3$ cannot be strongly irreducible, as we have constructed a frozen coloring in it.
\end{proof}

This last theorem together with the previous comments settle the existence of frozen colorings for $n=3$ and $n\ge 5$. To tackle the case of four colors, and in the process give an alternative proof of the lack of frozen colorings for $n\ge 5$, we will use a proposition from \cite{alon2019mixing} used in that paper to prove the lack of frozen $n$-colorings in $\mathbb{Z}^d$ for $n\ge d+2$.

\begin{proposition}[{\cite[Proposition~2.2]{alon2019mixing}}]\label{prop:nofroz_graph} For a graph $\Gamma$ let us define its \textit{edge-isoperimetric constant} by
	$$
	i_e(\Gamma)\coloneqq \inf_{F\subseteq \Gamma \text{ finite}}\frac{|E(F,\Gamma\backslash F)|}{|F|},
	$$
	where $E(F,\Gamma\backslash F)$ are the edges of $\Gamma$ connecting vertices from $F$ to $\Gamma\backslash F$. Denote by $\Delta$ the maximum degree of $\Gamma$. Then for every $n>\frac{1}{2}\Delta+\frac{1}{2}i_e(\Gamma)+1$ there do not exist frozen $n$-colorings of $\Gamma$.
\end{proposition}
\begin{theorem}\label{thm:no_frozen_n_ge_4} For $n\ge 4$ the GCS $\mathcal{C}_n$ does not admit a frozen coloring.
\end{theorem}
\begin{proof}
	Denoting by $\Gamma$ the Cayley graph of $\BS$ its maximum degree is $\Delta=4$, and if we prove that $i_e(\Gamma)=0$ we will have that, using proposition \ref{prop:nofroz_graph}, for $n>\frac{1}{2}\cdot 4+\frac{1}{2}\cdot 0+1=3$ this graph does not admit frozen $n$-colorings, proving the statement.
	
	Let us see that $i_e(\Gamma)=0$. Consider for every $k\ge 1: \ \gamma_{k}\coloneqq |E(R_k,\Gamma\backslash R_k)|$. We see that $\gamma_1=2N+2$, and that for every $k\ge 2$:
	$$
	\gamma_k=N^k+2+N(\gamma_{k-1}-N^{k-1})=2+N\gamma_{k-1},
	$$
	by using the fact that the $N$ sheets arising from the base of the rectangle $R_k$ are copies of the rectangle $R_{k-1}$. 
	%
	By solving this recursion we arrive at $\gamma_{m}=2\frac{N^{m+1}-1}{N-1}$, and with it give an estimate for the edge-isoperimetric constant:
	$$
	i_e(\Gamma)\le \liminf_{m\to \infty}\frac{|E(R_m,\Gamma\backslash R_m|)}{|R_m|}=\liminf_{m\to \infty}\frac{\gamma_m}{|R_m|}=\liminf_{m\to \infty}\frac{2}{mN^m}\frac{N^{m+1}-1}{N-1}=0.
	$$
	Hence $i_e(\Gamma)=0$, as we had claimed at the beginning of the proof.
\end{proof}


\subsection{Topological entropy}
We finish this section by focusing our attention on the topological entropy of the GCS $\mathcal{C}_n$. By estimating in how many different ways a pattern defined on a rectangle can be extended to a bigger one, we obtain the following bounds.
\begin{theorem}\label{thm:gcs_entropy_estimates} For $n\ge 3$, we have the following estimate for the topological entropy of the GCS $\mathcal{C}_n$:
	$$
	\log(n-2)\le \htop(\mathcal{C}_n)\le \log(n-1).
	$$
\end{theorem}
\begin{proof}
	Let us see that $\htop(\mathcal{C}_n)\le \log(n-1)$: a coloring of the rectangle $R_m$ may be extended to a coloring of the rectangle $R_{m+1}$ by coloring the remaining vertices having on each one at most $n-1$ options. With this
	\begin{align*}
	|\mathcal{L}_{R_{m+1}}(\mathcal{C}_n)|\le |\mathcal{L}_{R_m}(\mathcal{C}_{n})|(n-1)^{|R_{m+1}\backslash R_m|}.
	\end{align*}
	A simple calculation shows that $|R_{m+1}\backslash R_{m}|=N^m(mN+N-m)$. 
	With this:
	\begin{align*}
	\frac{1}{|R_{m+1}|}\log|\mathcal{L}_{R_{m+1}}(\mathcal{C}_n)|&=\frac{1}{(m+1)N^{m+1}}\log|\mathcal{L}_{R_{m+1}}(\mathcal{C}_n)|\\
	&\le \frac{1}{(m+1)N^{m+1}}\log|\mathcal{L}_{R_m}(\mathcal{C}_{n})|+\frac{N^m(mN+N-m)}{(m+1)N^{m+1}}\log(n-1)\\
	&=\frac{1}{N}\frac{m}{m+1}\frac{1}{mN^{m}}\log|\mathcal{L}_{R_m}(\mathcal{C}_{n})|+ \frac{mN+N-m}{m+1}\frac{1}{N}\log(n-1)\\
	&=\frac{1}{N}\frac{m}{m+1}\frac{1}{|R_m|}\log|\mathcal{L}_{R_m}(\mathcal{C}_{n})|+ \frac{m(N-1)+N}{m+1}\frac{1}{N}\log(n-1).
	\end{align*}
	Taking the limit $m\to \infty$ we arrive at
	\begin{align*}
	\htop(\mathcal{C}_n)&=\lim_{m\to \infty }	\frac{1}{|R_{m+1}|}\log|\mathcal{L}_{R_{m+1}}(\mathcal{C}_n)|\\
	&\le \lim_{m\to \infty }\frac{1}{N}\frac{m}{m+1}\frac{1}{|R_m|}\log|\mathcal{L}_{R_m}(\mathcal{C}_{n})| +\lim_{m\to \infty }\frac{m(N-1)+N}{m+1}\frac{1}{N}\log(n-1)\\
	&= \frac{1}{N}\htop(\mathcal{C}_n)+\frac{N-1}{N}\log(n-1),
	\end{align*}
	from where 
	$$\htop(\mathcal{C}_n)\le \log(n-1).$$

	Now let us see that $\htop(\mathcal{C}_n)\ge \log(n-2)$:
	the rectangle $R_m$ can be colored starting from the upper levels to the lower levels, ensuring at least $n-2$ color options at each element, since in this way each vertex of the Cayley graph has at most two neighbors already colored. Then this coloring can be extended to the whole graph by using Proposition \ref{prop:gcs_rectangle_extension} and hence forming a globally admissible configuration for $\mathcal{C}_n$. From this we arrive at
	\begin{align*}
	|\mathcal{L}_{R_m}(\mathcal{C}_{n})|\ge (n-2)^{|R_m|},
	\end{align*}
	and so $\htop(\mathcal{C}_n)\ge \log(n-2)$.
\end{proof}
In particular, the lower bound from the previous proposition shows us that $\htop(\mathcal{C}_n)>0$ for every $n\ge 4$, but gives us no new information for the case of three colors $\mathcal{C}_3$, since by definition $\htop(\mathcal{C}_3)\ge 0$. In the next proposition we show that indeed $\htop(\mathcal{C}_3)>0$, dividing the proof in two cases depending on the parity of $N$. If $N$ is odd we can exploit the fact that the Cayley graph of $\BS$ is bipartite (we know this since we have already constructed a $2$-coloring of it in Proposition \ref{prop:GCS_nonemptiness}, or equivalently observing that it has no odd cycles) to give a straightforward proof, while in the case of even $N$ we give a more delicate construction to arrive at the same result.
\begin{theorem}\label{thm:C_3_has_positive_entropy} The GCS $\mathcal{C}_3\subseteq \{0,1,2\}^{\BS}$ has positive topological entropy.
\end{theorem}
\begin{proof}
	Let us start with the case of odd $N$. Here the Cayley graph of $\BS$ is bipartite and hence so is every rectangle $R_m$, for $m\ge 1$. Consider a \textit{partition} of $R_m$ into two sets $A$ and $B$, meaning that all edges of the graph are composed of a vertex in $A$ and a vertex in $B$. Then one of them, which we take to be $A$ without loss of generality, must have cardinality at least $\frac{1}{2}|R_m|$. Then we can create proper colorings of $R_m$ by coloring the vertices of $B$ with one color and have the freedom to choose between two colors for every vertex of $A$, and then extend this pattern on $R_m$ to the rest of the group as was said earlier.
	
	With the above we can estimate a lower bound for the number of proper colorings of $R_m$ by
	$$
	|\mathcal{L}_{R_m}(\mathcal{C}_3)|\ge 2^{|A|}\ge 2^{\frac{1}{2}|R_m|}.
	$$
	Then taking logarithm and dividing by $|R_m|$ we arrive at
	$$
	\frac{1}{|R_m|}\log |\mathcal{L}_{R_m}(\mathcal{C}_3)| \ge \frac{1}{2}\log (2),
	$$
	to finally take limit as $m\to \infty$ and obtain
	$$
	\htop(\mathcal{C}_3)\ge \frac{1}{2}\log(2)>0,
	$$
	which is what we wanted.

	Now consider the case of even $N$. For $m\ge 1$ we define a pattern $p\in \mathcal{A}^{R_{2m}}$, where some of the vertices of this rectangle will have the freedom of being assigned the symbol $1$ or $2$, in either case remaining a proper $3$-coloring of $R_{2m}$.
	
	First let us define 
	$$
	p|_{[e_{\BS},a^{N^{2m}-1}]}=(0(12)^{N^2/2})^{\infty}|_{[0,N^{2m}-1]}.
	$$
	That is, we are putting in the elements $\{a^{0},a^{1},\ldots,a^{N^{2m}-1}\}$ an initial segment of the infinite $(N^2+1)$-periodic sequence $\left(0(12)^{N^2/2}\right)^{\infty}$. Moreover, we will copy this sequence upwards into even levels: for $k\in \{0,\ldots,N^{2m}-1\}$ and $0\le i<m$ we define
	$$
	p_{a^{k}b^{2i}}=p_{a^k}.
	$$
	
	Note that $a^kb^{2i}a=a^{k+N^{2i}}b^{2i}$ and that $N^{2i}=(N^{2})^i=(-1)^{i}\ \mathrm{mod}(N^2+1)$. This means that at level $2i$ we see either (a shift of) the sequence $(0(12)^{N^2/2})^{\infty}$ or $(0(21)^{N^2/2})^{\infty}$. In particular, up until now $p$ satisfies the rules for being a proper coloring.
	
	It remains to define $p$ at odd levels. Fix $0\le i<m$, and consider the level $2i+1$. We will define $p$ at this level so that the value $p_{a^{(N^2+1)k}b^{2i+1}}$, $0\le k<\frac{N^{2m}}{N^2+1}$, can take any value in $\{1,2\}$ and still define a proper coloring. In other words, the element $a^{(N^2+1)k}b^{2i+1}$ will have all its neigbors assigned the color $0$.
	
	Indeed, define for $i$ and $k$ as above, for $0\le \ell <N^2+1$, and for any choice $\alpha(i,\ell,k)\in\{1,2\}$:
	$$
	p_{a^{(N^2+1)k+\ell}b^{2i+1}}=\begin{cases}
	\alpha(i,\ell,k), &\text{ if }\ell=0,\\
	0 \ &\text{ if } \ell=N \text{ or }\ell=N^2+1-N,\\
	1\ &\text{ if } \ell\notin\{N,N^2+1-N\} \text{ and }\ell \text{ is even,}\\
	2\ &\text{ if } \ell\notin\{N,N^2+1-N\} \text{ and }\ell \text{ is odd}.
	\end{cases}
	$$
	Since $$p_{a^{(N^2+1)k+\ell}b^{2i}}=\begin{cases}
	1 &\text{ if }\ell \text{ is odd}\\
	2 &\text{ if }\ell \text{ is even,}\\
	\end{cases}$$
	for $i$, $k$ as above, and for $0<\ell<N^2+1$,
	we see that $p$ defines a proper coloring in $\mathcal{C}_3$.

	\begin{figure}
		\centering
		\includegraphics[]{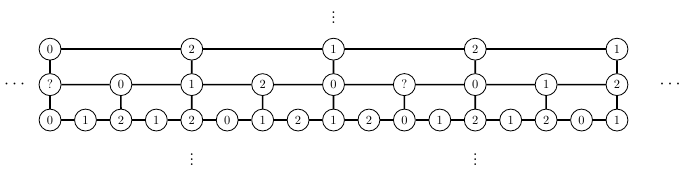}
		\caption{Nodes with the symbol ``?'' have freedom of choosing between a symbol ``1'' or ``2'', since all their neighbors have ``0''s in them}
	\end{figure}
	
	The above shows that
	$$
	|\mathcal{L}_{R_{2m}}(\mathcal{C}_3)|\ge 2^{m\left\lfloor \frac{N^{2m}}{N^2+1}\right\rfloor },
	$$
	and hence
	\begin{align*}
	\frac{1}{|R_{2m}|}\log |\mathcal{L}_{R_{2m}}(\mathcal{C}_3)|&\ge\frac{m\left\lfloor \frac{N^{2m}}{N^2+1}\right\rfloor}{2mN^{2m}}\log(2) \\
	&\ge \frac{\frac{N^{2m}}{N^2+1}-1}{2N^{2m}}\log(2).
	\end{align*}
	Finally, taking limit as $m\to\infty$ we obtain
	$$
	\htop(\mathcal{C}_3)\ge \frac{1}{2(N^2+1)}\log(2)>0.
	$$
\end{proof}	

\section{Further questions and possible generalizations}
\label{section: further questions}

It is interesting to think about how the properties from Section \ref{section:graph_coloring_subshifts} for graph-coloring subshifts in $\BS$ generalize to the more fundamental context of ascending HNN-extensions $G*_{\psi}$ of finitely generated abelian groups. The answers will now probably be separated in two cases depending on the nature of the Cayley graph of the group $G$ with respect to the chosen generating set, depending on whether it is bipartite or not. Results will also depend on the value of the index $|G:\psi(G)|$, since this number will play a role in the length of the new cycles that arise in the Cayley graph of $G*_{\psi}$, in contrast with those found in $G$.

The method used to find the topological entropy bounds for GCS in $\BS$ was counting patterns using rectangles, and thus should be able to be replicated to some extent for $G*_{\psi}$, using the natural generalizations of rectangles to this group. 
\begin{problem*}
	How do the topological entropy related results from Theorem \ref{thm:summary_gcs_entropy} generalize to GCS in $G*_{\psi}$?
\end{problem*}
With respect to frozen colorings, Proposition \ref{prop:nofroz_graph} was fundamental for proving the absence of such configurations for $n\ge 4$ and it will play the same role in the case of ascending HNN-extensions $G*_{\psi}$ of a finitely generated abelian group, by using the generalization of rectangles to these groups. On the other hand, the existence of frozen $3$-colorings shown in Theorem \ref{thm:existence_frozen_3} was somehow involved and it does not generalize immediately to the case of $G*_{\psi}$.
\begin{problem*}
	Find explicit constructions of frozen $3$-colorings on $G*_{\psi}$.
\end{problem*}

A final remark is that most of the results obtained in Section \ref{section:graph_coloring_subshifts} depend only on the number of colors $n$, and not on the parameter $N$ of $\BS$. More generally, we ask the following.
\begin{problem*}
	Consider a finitely generated abelian group $G$, two different (non-trivial, non-surjective) monomorphisms $\psi, \psi':G\to G$, and the corresponding ascending HNN-extensions $G*_{\psi}$ and $G*_{\psi'}$. For $n\ge 3$, what are the dynamical differences between the corresponding GCS associated to $G*_{\psi}$ and $G*_{\psi'}$? In particular, do they have equal topological entropy?
\end{problem*}

	\textbf{Acknowledgements}. I would like to thank my Master's advisor Michael Schraudner for his continuous support, and Nathalie Aubrun for her helpful comments and ideas. This work was developed as part of a Master's degree thesis at Universidad de Chile, partially supported by CONICYT-PFCHA/Mag\'ister Nacional/2019 - 22190176. 
	
\bibliographystyle{plain}
\bibliography{bibliography}{}
\end{document}